\documentclass[12pt]{amsart}
\usepackage{amssymb,amsmath,amsfonts,latexsym,setspace}
\usepackage{bm}
\theoremstyle{plain}
\newtheorem{propn}{Proposition}[section]
\newtheorem{thm}[propn]{Theorem}
\newtheorem{lemma}[propn]{Lemma}
\newtheorem{cor}[propn]{Corollary}
\newtheorem*{thm*}{Theorem}

\theoremstyle{definition}
\newtheorem*{defn}{Definition}

\theoremstyle{remark}
\newtheorem*{rem}{Remark}
\newtheorem*{rems}{Remarks}

\newcommand{\Hil}{\mathcal{H}}

\newcommand{\norm}[1]{\lVert #1 \rVert}

\newcommand{\ov}[1]{\overline{#1}}

\newcommand{\Dt}{D_{T}}
\newcommand{\dt}{\Delta_T}

\newcommand{\cpt}{\Psi_T}

 \newcommand{\dst}{\mathcal{D}_T}
 \newcommand{\ds}[1]{\mathcal{D}_{#1}}

\newcommand{\q}{\mathcal{Q}}
\newcommand{\s}{\mathcal{S}}

\newcommand{\Nat}{\mathbb{N}}
\newcommand{\Comp}{\mathbb{C}}

\newcommand{\Fockd}{\mathcal{F}^{2}_{d}}

\newcommand{\ot}{\otimes}

\DeclareMathOperator{\spa}{span}
 
\DeclareMathOperator{\ran}{ran}

 \newcommand{\rano}{\ov{\ran}}

\begin{document}

\title{Maximal contractive tuples}

\author[Das] {B. Krishna Das}

\address{
(B. K. Das) Indian Statistical Institute \\ Statistics and
Mathematics Unit \\ 8th Mile, Mysore Road \\ Bangalore \\ 560059 \\
India}

\email{dasb@isibang.ac.in}

\author[Sarkar]{Jaydeb Sarkar}

\address{
(J. Sarkar) Indian Statistical Institute \\ Statistics and
Mathematics Unit \\ 8th Mile, Mysore Road \\ Bangalore \\ 560059 \\
India}

\email{jay@isibang.ac.in, jaydeb@gmail.com}

\author[Sarkar]{Santanu Sarkar}

\address{
(S. Sarkar)  Department of Mathematics\\
  Indian Institute of Science\\
  Bangalore\\
  560 012\\
  India}

\email{santanu@math.iisc.ernet.in}


 \begin{abstract}
  Maximality of a contractive tuple of operators is considered.
  Characterization of a contractive tuple to be maximal is obtained.
  Notion of maximality of a submodule of Drury-Arveson module on 
  the $d$-dimensional unit ball $\mathbb{B}_d$
  is defined. For $d=1$, it is shown that
  every submodule of the Hardy module over the unit disc is maximal. 
  But for $d\ge 2$ we prove that any homogeneous submodule or
  submodule generated by polynomials is not maximal.
  A characterization of a submodule to be maximal 
  is obtained.
 \end{abstract}

\maketitle

\section{Introduction}
 Let $T=(T_1,\dots,T_d)$ be a $d$-tuple of bounded linear
 operators on some Hilbert
 space $\Hil$. We say that $T$ is a \emph{row contraction}, or,
 \emph{contractive tuple} if the row operator $(T_1,\dots, T_d):
 \Hil^d\to \Hil$ is a contraction or equivalently
 $\sum_{i=1}^d T_iT_i^*\le I_{\Hil}$. The \emph{defect operator}
 $D_T:=(I-\sum_{i=1}^d T_iT_i^*)^{1/2}$ and the \emph{defect dimension}
 $\dt:=\dim[\ov{\ran}D_T]$ associated with the contractive tuple $T$
 is an important invariant in operator theory. For instance,
 a pair of shift operators are unitary equivalent if and only if
 the defect dimensions are the same. The same result holds true
 for $d$-tuple of pure isometries with orthogonal ranges (\cite{popescu}).
 In order to extract more information about contractive tuples
 one can proceed further to form a sequence of defect indices (defined below).

The defect sequence for contractive tuple and the notion of
 maximality of a contractive tuple  was introduced in ~\cite{GaW}
 for $d=1$ case. In a recent paper this notion was extended for
 $d$-tuple of operators (\cite{defect}). The main aim of this paper is
 to characterize maximal contractive tuples in the commuting as well as
 non-commuting case. In the non-commuting setup it turns out that the
 restriction of creation operators on the full Fock space
 to an invariant subspace is always maximal
 but in the commuting setup same conclusion does not hold.
 Examples of submodules of the Drury-Arveson module are given to illustrate the
 above fact and a characterization of a submodule to be maximal is
 also obtained.

 The plan of the paper is as follows. After
 introducing the completely positive map associated
 to a contractive tuple we define defect sequence
 and obtain its properties in Section 2.
 In Section 3, we provide a characterization for
 maximal contractive tuples and consequently establish
 some relations between minimal function of a particular type of single pure
 contraction and the dimension of the Hilbert space
 on which the contraction acts.
 In the last section we investigate the maximality for the tuple
 $( M_{z_1}|_{\s}, \cdots, M_{z_d}|_{\s} ) $ where $\s$ is a
 proper submodule of Drury-Arveson module and the tuple
 $(M_{z_1},\cdots, M_{z_d})$ is the $d$-shift of the
 Drury-Arveson module.

\section{Defect sequence}

 In this section we define the notion of the
 defect sequence of a tuple of contraction and
 study its properties. Some of the result can be
 found in (\cite{defect}) and we include proofs of them as
 it uses different but simple method.
 We fix for this section a contractive $d$-tuple
 $T=(T_1,\dots,T_d)$ of operators acting on a Hilbert space
 $\Hil$ in which the tuple $T$ is not necessarily commuting
 and the Hilbert space $\Hil$ is infinite dimensional
 in general unless otherwise we specify it.

 We begin with defining the completely positive map
 associated to the contractive tuple $T$ as follows:
  \begin{equation}
 \label{cp map}
 \cpt: B(\Hil)\to B(\Hil),\quad X\mapsto \sum_{i=1}^d T_iXT_i^*.
 \end{equation}
 This map is very essence
 for simplifying the study of defect sequences.
  The following decreasing chain of operator inequality
  \[
   I\ge \cpt(I)\ge \cpt^2(I)\ge \dots
  \]
  is immediate from the contractivity of the tuple $T$.
  The contractive tuple $T$ is said to be \emph{pure} if $\cpt^n(I)\to 0$
  in the strong operator topology (S.O.T.) as $n\to\infty$.

 The following rule of multiplication of operator tuples
 is in use. Let $\Lambda=\{1,\dots,d\}$.
 For $n\in\Nat$, we denote $T^n$ by the following $d^n$-tuple of operators
 \[T^n = (T_{i_1}T_{i_2}\dots T_{i_n}:  i_{j}\in\Lambda, j=1,\dots,n),\]
 and for $n=1$ we set $T^1=T$. In particular, for $n=2$,
 $T^2$ is the following $d^2$-tuple
 \[(T_1^2, T_1 T_2,\dots,T_1T_d,T_2T_1, T_2^2,
 \dots,T_d^2): \Hil^{d^2}\to \Hil.\]
 Under this rule of multiplication
 note that $\Psi_{T^{2}} (X) = \cpt(\cpt(X))= \Psi^2_{T} (X)$
 for all $X\in B(\Hil)$, where $\cpt$ is as in ~\eqref{cp map}.

 \begin{defn}
 The defect operator of $T$, denoted by $\Dt$,
 is the bounded linear operator on $\Hil$ defined by
 $$\Dt := (I- \sum_{i=1} ^d T_i T_i^*)^{1/2}
 = (I-\cpt(I))^{1/2}.$$ The \emph{first defect index} $\dt$ is the dimension of
 the \emph{first defect space}
  $\dst$ where \[\dst :=\rano\Dt =\rano D_T^2= \rano(I- \cpt(I)).\]
 The \emph{$n$-th defect index} of the tuple $T$
 is the dimension of the \emph{$n$-th defect space}
 $\rano D_{T^{n}}$ where \[D_{T^{n}}^2 = I- \Psi_{T^n} (I)
 = I-\Psi^n_{T} (I).\]
 \end{defn}
 We denote by  $\ds{n}$ the
 $n$-th defect space of a contractive tuple $T$,
 where context dictate the tuple $T$.
 Here we note the following identity
 \begin{align}
 \label{sum formula}
   I-\cpt^n(I)
   &=[I-\cpt(I)]+ \cpt[I-\cpt(I)]+\cdots+\Psi_T^{n-1}[I-\cpt(I)]
   \nonumber\\
   &=\sum_{i=0}^{n-1} \cpt^i (I-\cpt(I))
 \end{align}
 The properties of the defect sequence are as follows.
 \begin{propn}
 \label{properties}
 \textup{(i)} Defect spaces of $T$ are increasing subspaces of $\Hil$,
 that is, for $n\le k$, $\ds{n}\subset\ds{k}$. \\
 \textup{(ii)}$\Delta_T^n \leq \Delta _T^k$, for all $n \leq k$.\\
 \textup{(ii)} For $n\in\Nat$,  $\Delta_T^n
 \leq (1+d+d^2+\cdots+d^{n-1})\Delta_T $.
 \end{propn}

 \begin{proof}
 (i) For $k\geq n,$ by row contractivity of $T$ we have
 \[I\geq\cpt^n (I)\geq \cpt^k(I).\]
 Therefore, \[0\leq I-\cpt^n(I)\leq I-\cpt^k(I),\] and
 consequently \[\rano(I-\cpt^n(I))\subseteq \rano (I-\cpt^k(I)).\]
 Thus (i) follows. (ii) follows immediately from (i).

 \noindent (iii) The result follows from ~\eqref{sum formula} and the fact that
 \[\dim [\rano~\cpt^l (I-\cpt(I))] \leq d^l \dt,\] for all $l\in \Nat.$
 \end{proof}

 \begin{rems}
  (i) If $T$ is a single contraction, that is, if
  $d=1$ then $\dt^n\le n\dt $ for all $n\in\Nat$ (\cite{GaW}).\\
  (ii) If $T$ is a commuting $d$-tuple then
  $\dt^n\le  \big(\sum_{k=0}^{n-1}
\binom{ k+d-1}{ d-1}\big) \dt$.
 \end{rems}
 Before we provide the explicit expression of the defect spaces
 we need the following lemma.
 \begin{lemma}
 \label{key for the expression}
 For each $n\in\Nat$, $T|_{\ds{n}^d} : \ds{n}^d \to \ds{n+1}$,
 where $\ds{n}^d$ is the direct sum of $d$ copies of $\ds{n}$.
 \end {lemma}

 \begin{proof}
 First note that,
 \begin{align*}
 \sum_{i=1}^dT_i (I- \cpt^n(I))^{1/2} (I- \cpt^n(I))^{1/2} T_i^*
  &= \sum_{i=1}^dT_iT_i^* - \sum_{i=1}^dT_i\cpt^n(I) T_i^*\\
  &
 \leq I- \cpt^{n+1}(I).
 \end{align*}
 Letting $$R: =(T_1,\dots,T_d) \begin{pmatrix}
 D_{T^n}& 0 & 0 & \cdots & 0 \\
  0 & D_{T^n} & 0 & \cdots & 0 \\
   \vdots  & \vdots  & \ddots & \vdots  & \vdots \\
   0 & 0 & 0 & \cdots & D_{T^n} \\
  \end{pmatrix}_{d\times d},$$
  we have from the above inequality that
  $RR^*\le I-\cpt^{n+1}(I)$.
 Thus \[\ran R
 \subseteq \ran(I- \cpt^{n+1}(I)),\] and this completes the proof.
\end{proof}

 \begin{rem}
 Simple induction argument shows that
 $T^{n}|_{\ds{l}^{d^n}} : \ds{l}^{d^n} \to \ds{l+n}$ for all
 $l,n \in \mathbb{N}.$
 \end{rem}
 The following expression of defect spaces in terms of first defect space 
 is quite useful to compute defect spaces
 and used throughout the paper.
 \begin {propn}
 \label{representation of defect space}
 The defect spaces of a contractive tuple $T$ has the following form:
 $$\ds{n} = \ds{1} \vee T (\ds{1} ^d) \vee T^{ 2} (\ds{1}^{d^2}) \vee
 \cdots \vee T^{ n-1} (\ds{1} ^{d^{n-1}})$$
 for all $n\in\Nat$.
 \end{propn}

 \begin{proof}
 $\ds{n} \subseteq \ds{1} \vee T (\ds{1} ^d) \vee T^{2} (\ds{1}^{d^2})
 \vee \cdots \vee T^{ n-1} (\ds{1} ^{d^{n-1}})$ follows
 from ~\eqref{sum formula} and the other inclusion follows from the previous lemma.
\end{proof}
 \begin{rem}
 Let $n< m$. Then as $I-\cpt^m=(I-\cpt^n)+ \cpt^n(I-\cpt^{m-n}(I))$, we
 have $\ds{m}\subseteq \ds{n}\vee T^{n}(\ds{m-n}^{d^n})$. The other
 inclusion follows from the remark after Lemma~\ref{key for the expression}. Thus
 $\ds{m}=\ds{n}\vee T^{ n}(\ds{m-n}^{d^n})$.
 \end{rem}

 \begin{cor}
 \label{stable}
 If $\dt^n=\dt^{n+1}$ for some $n\in\Nat$ then $\dt^n=\dt^m$
 for all $m>n$.
 \end{cor}
 \begin{proof}
 If $\dt^n=\dt^{n+1}$ for some $n\in\Nat$ then
 $\ds{n}=\ds{n+1}$. Note that $\ds{n+2}=\ds{n+1}\vee T^{n+1}(\ds{1}^{d^{n+1}})$
  and $T^{n+1}(\ds{1}^{d^{n+1}})= T\Big(\big(T^{n}(\ds{1}^{d^n})\big)^{d}\Big)$.
  Since $T^{n}(\ds{1}^{d^n})\subset \ds{n+1}=\ds{n}$ we have $\ds{n+2}=\ds{n}$.
  Thus an induction argument
 gives the result.
 \end{proof}

 The above two propositions are from ~\cite{defect}, Theorem 2.2 and Theorem 2.4
 but the method used here will illuminate further studies in this direction.

 \section{Maximal tuple of operators}
 In this section we study the notion of maximality of contractive
 tuples.
 The necessary and sufficient condition
 for a contractive tuple to be maximal is obtained.
 For this section we always assume
 that the first defect dimension $\dt$ of a contractive tuple $T$ is finite.

 The following set of notation is used throughout
 this section. Let $\Lambda=\{1,2,\dots,d\}$ be a fixed index set.
 For every $k\in\Nat$, let $F(k,\Lambda)$ be the
 set of all functions from $\{1,2,\dots,k\}$ to $\Lambda$, and set
 \begin{equation}
 \label{F}
  F:=\cup_{k=0}^{\infty}F(k,\Lambda), \quad F_{n]}:=\cup_{k=0}^{n}F(k,\Lambda)
 \end{equation}
 where $F(0,\Lambda)$ stands for $\{0\}$.
 For $T=(T_1,T_2,\dots,T_d)$, a $d$-tuple of operators,
 and $f\in F(k,\Lambda)$, we denote
 \begin{equation}
 \label{operator composition}
 T_{f}=T_{f(1)}T_{f(2)}\dots T_{f(k)}\quad
  \text{ and } T_{0}=I.
  \end{equation}

 \begin{defn}
 A contractive $d$-tuple $T$ on
 a Hilbert space $\Hil$
 is called \emph{maximal} if
 \[\dt^n=(1+d+\dots+d^{n-1})\dt\] for all $n\in\Nat$.
 If $\Hil$ is finite dimensional, then $T$ is maximal
 if
 \[
  \dt^n=\left\{\begin{array}{cl}(1+d+\dots+d^{n-1})\dt&
  \quad \text{if } (1+d+\dots+d^{n-1})\dt\le \dim\Hil,\\
  \dim\Hil & \quad \text{otherwise}
              \end{array}\right.
 .\]
 \end{defn}
 The maximality of a commuting contractive $d$-tuple is defined in the same way
 replacing the number $(1+d+\dots+d^{n-1})$ by $\sum_{k=0}^{n-1}
\binom{ k+d-1}{ d-1}$ in the above definition.

 It is clear from the properties of defect sequence that if $\dt^n=
 (1+d+\dots+d^{n-1})\dt$ (or $\dt^n=(\sum_{k=0}^{n-1}
 \binom{ k+d-1}{ d-1})\dt$ in the commuting case) for some $n\ge 2$, then
 $\dt^l=(1+d+\dots+d^{l-1})\dt$ (respectively, $\dt^l=(\sum_{k=0}^{l-1}
 \binom{ k+d-1}{ d-1})\dt$) for all $l\le n$. Thus
 for a non-commuting or commuting tuple, once the sequence
 of numbers $\dt^n$ departs from the sequence of maximal
 possible values, it never returns.
 \begin{rem}
 For a non-commuting (commuting) contractive tuple $T$ on an infinite
 dimensional Hilbert space $\Hil$,
 let $\dt=n$ and  $\{\xi_i:i=1,\dots,n\}$ be a
 basis of $\mathcal{D}_1$. Then
 $T$ is maximal if and only if the set
 \[\{T_{f}\xi_i: f\in F, i=1,\dots,n\},\]
 (respectively, $\{T_1^{n_1}\dots T_d^{n_d}\xi_i:
 n_1,\dots, n_d\in\Nat, 1 \leq i \leq  n\}$)
 is linearly independent,
 where $T_f$ is as in ~\eqref{operator composition}.
 \end{rem}

 A single contraction $T$ acting on a Hilbert space $\Hil$
 with $\dt=1$ is maximal if $\dt^n=n$ for all $n\le\dim\Hil$. The
 next theorem provides a large class of contractions which are maximal.

 \begin{thm}
 \label{pure is maximal}
 Let $T$ be a single pure contraction on $\Hil$ with $\Delta_T = 1.$
 Then $\Delta_{T^n} = n, $ for $0\leq n \leq \dim\Hil.$
 \end{thm}

 \begin{proof}
 Since $T$ is a pure contraction with $\Delta_T = 1,$
 so $T$ is unitary equivalent to the operator
 $P_{H_\theta} M_z |_{H_\theta}$ where
 $H_{\theta}= H^2(\mathbb{D})\ominus \theta H^2(\mathbb{D})$
 is a co-invariant subspace of the Hardy space $H^2(\mathbb{D})$
 on the unit disc and $\theta \in H^\infty(\mathbb{D})$ is an inner function.
 Then it is enough to prove the theorem for the contraction 
 $R=P_{H_\theta} M_z |_{H_\theta}$.
 A simple calculation reveals that  $D_R=P_{H_\theta} P_{\mathbb C} P_{H_\theta}$
 and as $\Delta_T =\Delta_R= 1,$   we have $P_{\mathbb C} P_{H_\theta} \neq 0 $
 and $\ran (P_{\mathbb C} P_{H_\theta}) = \mathbb{C}.$
 Note that (cf. \cite{Ber}) 
 \[P_{H_\theta} (1) =(I-P_{H_{\theta}^{\perp}})1= 1- \overline{\theta (0)}
 \theta.\]
 Then the first defect space of the operator
 $P_{H_\theta} M_z |_{H_\theta}$ is
 \[\ds{1} = \spa \{1 - \ov{\theta (0)} \theta \}.\]
 By the following elementary calculation we have
 \[(P_{H_\theta} M_z |_{H_\theta}) (1- \ov{\theta (0)} \theta)
 = P_{H_\theta} (z - \ov{\theta (0)}  z \theta) = 
 (I-P_{H_{\theta}^{\perp}})z=
 z - ( \ov{\theta (0)}z +  \overline{\theta^{'} (0)}) \theta.\]
 Then by Proposition~\ref{representation of defect space},
\[\ds{2} = \spa \{   1- \ov{\theta (0)}\theta , z - (
\ov{\theta (0)}z +  \overline{\theta^{'} (0)} ) \theta \}.\] An easy
induction argument yields
\[(P_{H_\theta} M_z ^{n} |_{H_\theta}) (1-
\ov{\theta (0)}\theta) =(I-P_{H_{\theta}^{\perp}})z^n=
z^n - \big(\ov{\theta(0)}z^n + \ov{\theta^{'} (0)} z^{n-1}
+ \cdots + \ov{\theta^{(n)} (0)}\big) \theta,\]for all $n \geq 0$.
Therefore, one has the explicit expression of $\ds{n}$ as follows. Let us
denote
$$v_i = z^i - \big(\ov{\theta(0)}z^i + \ov{\theta^{'} (0)} z^{i-1}
+ \cdots + \ov{\theta^{(i)} (0)}\big) \theta\quad (i\in\Nat).$$
Then $\ds{n}=\spa\{v_i:i=1,\dots, n\}$. Now suppose that
$\dim(\ds{l})=\dim(\ds{l+1}) $ for some $l\in\Nat$. Then as $\Delta_R=1$
and defect sequence is an increasing sequence it suffices to prove that
$H_{\theta} =  \ds{l}.$ For a contradiction let $f \in
H_{\theta}\ominus \ds{l}$. Then for all $i\in\Nat$, $\langle f , \theta z^{i} \rangle =
0,$ and $\langle f ,v_i \rangle = 0$ together implies that
 $\langle f,z^i\rangle=0$ for all $i$.
 Thus $f=0$ and the proof follows.
\end{proof}
 The above result is due to ~\cite{GaW}, Theorem 1.4.
 However, our proof is different and more analytic and explicit.

 In case of a tuple of operators the above theorem is not true.
 The example of a pure tuple $T$ with $\dt=1$ but is not maximal
  can be found in ~\cite{defect}.

 Set $\ds{\infty}:=\cup_{n=1}^{\infty}\ds{n}$, where $\ds{n}$'s
 are the defect spaces of $T$.
 The multi-variable analogue of the previous theorem is as follows.
 \begin{propn}
 \label{pure}
 Let $T$ be a pure contractive $d$-tuple of operators on a Hilbert
 space $\Hil$.
 Then $\Hil=\ds{\infty}$.
 \end{propn}

 \begin{proof}
 First note that
 \begin{equation*}\ds{\infty}^{\perp}=
 \mathop{\cap}_{n \geq 1}\ds{n}^{\perp}=
 \mathop{\cap}_{n \geq 1}\ker(I-T^{ n}T^{ n *}).
 \end{equation*}
 Therefore, if $x\in \ds{\infty}^{\perp}$ then $\norm{x}=\norm{T^{ n *}x}$
 for all $n\in\Nat$. Since $T$ is pure we have $\cpt^n(I)\to 0$ in S.O.T.
 as $n \to \infty$. In particular,
 $\langle x, \cpt^n(I)x\rangle=\norm{(T^n)^{*}x}^2\to 0$
 as $n\to\infty$. We have thus obtained that $x=0$ concluding the proof.
 \end{proof}

 \begin{cor}
  Let $T$ be a pure contractive tuple of operators on an infinite
  dimensional Hilbert space $\Hil$. Then $\dt^m\neq\dt ^n$
  for $m\neq n$.
 \end{cor}
 \begin{proof}
 Let $\dt^m=\dt^n$ for $m<n$. We know  $\ds{m}\subset\ds{n}$ and
 they have same finite dimension implies $\ds{m}=\ds{n}$.
 Then by Corollary~\ref{stable} $\ds{k}=\ds{m}$ for all $k\ge m$. Thus
 $\ds{\infty}=\ds{m}$ and is of finite dimension, which
 is a contradiction.
 \end{proof}

 Now we provide a characterization of maximal contractive tuples.

 \begin{thm}
 \label{characterisation of tuple}
  Let $T$ be a contractive $d$-tuple acting on an infinite dimensional Hilbert
  space $\Hil$ such that $\dt=1$. Then the following are equivalent:\\
  \textup{(i)} $T$ is maximal.\\
  \textup{(ii)} There is no polynomial $P$ of $d$ non-commuting variables
  such that
  \[
   P(T_1,\dots,T_d)|_{\ds{1}}=0.
  \]
 \end{thm}

 \begin{proof}
  Since $\dt=1$, let $\ds{1}=\Comp\xi$ for some $\xi\in\Hil$. Then by
  Proposition~\ref{representation of defect space}
  \[
   \ds{n}=\spa\{T_{f}\xi: f\in F_{n-1]}\}
  \]
 where $T_f$ is as in ~\eqref{operator composition}. Note that
 $|F(n,\Lambda)|=d^n$. Thus $T$ is maximal if and only if for
 all $n\in\Nat$,
 the set of vectors $\{T_f\xi: f\in F_{n-1]}\}$
 are linearly independent. Now it is clear that sets of the
 above type are linearly independent if and only if (ii) holds.
 This concludes the proof.
 \end{proof}

 \begin{rems}
 (i) A commuting contractive $d$-tuple $T$ on an infinite dimensional Hilbert space 
 $\Hil$ with $\dt=1$
 is maximal if and only if \[P(T_1,\dots,T_d)|_{\ds{1}}\neq0,\]
 for any polynomial $P$ of commuting $d$ variables.\\
 (ii) Let $M_{z_i}, i=1,\dots,d,$ denote the multiplication operators 
 on the Drury-Arveson module (see \cite{arveson}, \cite{D} or Section 4) $H^2_d$
 by co-ordinates. Then consider the tuple
 $$M=(P_{\q}M_{z_1}|_{\q},\dots,P_{\q}M_{z_d}|_{\q}),$$ where
 $\q$ is a quotient module of $H^2_d$ given by 
 $\q=H^2_d\ominus \theta H^2_d$ and $\theta$ is a multiplier.
 Let $\Delta_M=1$ and $\q$ is infinite dimensional. 
 Then note that $\theta f\in\q^{\perp}$ for any $f\in H^2_d$
 and this implies $\theta(M)=0$. Therefore by the first remark
 $M$ is maximal if and only if $\theta$ is not a polynomial.

 \end{rems}

 \begin{cor}
  Let $T$ be as in the above theorem. Then the following are equivalent:\\
  \textup{(i)} $\dt^n=1+d+d^2+\dots+d^{n-1}$ for all $n\le m$
  and $\dt^n<1+d+\dots+ d^{n-1}$ for all $n>m$.\\
 \textup{(ii)} There is no non-commuting polynomial $P$ with $d$ variable of degree
 less than $m$ such that $P(T)|_{\ds{1}}=0$ and there is a
 non-commuting polynomial $Q$ with $d$ variable of degree $m$
 such that $Q(T)|_{\ds{1}}=0$.
 \end{cor}
 \begin{proof}
  By the same argument as in the proof of the above theorem
  it follows that the dimension of $\ds{n}$ is maximal for
  some $n$ if and only if there is no polynomial $P$ of degree
  smaller than $n$ such that $P(T)|_{\ds{1}}=0$. This completes the proof.
  \end{proof}
 For a single contraction acting on a finite dimensional Hilbert space $\Hil$,
 we have the
 following immediate corollary.

 \begin{cor}
  \label{single contraction}
 Let $T$ be a single contraction acting on a finite
 dimensional Hilbert space $\Hil$ with $\dt=1$.
 Then the following are equivalent:\\
 \textup{(i)}
 \[ \dt^n=\left\{\begin{array}{rl}
                  n, & n\le m\le \dim\Hil\\
                  m, & n>m
                 \end{array}\right. .\]
 \textup{(ii)} The degree of the minimal polynomial
 of $T$ is at least $m$ and there is a polynomial $P$ of
 degree $m$ such that $P(T)|_{\ds{1}}=0$.
 \end{cor}

 Now we recall some of the work of Popescu (\cite{popescu}) in order to
 characterize pure maximal tuple of operators.
 We denote by $\Fockd$ the full Fock space over the $d$-dimensional
 Hilbert space $\Comp^d$ with orthonormal basis
 $(e_1,e_2,\dots ,e_d)$ . It is often represented by
 \[
  \Fockd=\Comp\oplus_{m\ge 1}(\Comp^d)^{\otimes m}
 \]
 but we use the notation ~\eqref{F} to describe it and
 simplify notation as follows. If $f\in F(k,\Lambda)$, let
 \[
  e_{f}=e_{f(1)}\otimes e_{f(2)}\otimes \dots\otimes e_{f(k)},\ \text{ and for }
  k=0,\ e_0=\omega.
 \]
 We call $\omega$ the vacuum vector.
 Then $\Fockd$ is the Hilbert space with basis $\{e_{f}:f\in F\}$.
 For each $n \in \mathbb{N}$, the $n$-th particle space is denoted by $\Gamma_{n]}$
 and is defined by \[\Gamma_{n]}:=\spa\{e_{f}: f\in F_{n]}\}.\]

 \textit{Creation operators} on the full Fock space $\Fockd$ is denoted by
  $S_i, i=1,\dots, d$ and defined by
  \[
    S_i:\Fockd\to\Fockd,\quad \psi\mapsto e_i\otimes\psi,\quad (i=1,\dots,d) .
  \]
   The complete characterization of invariant
  subspaces (consequently co-invariant subspaces) for these creation
  operators on full Fock space by Popescu
  (see \cite{Pop1989}, \cite{popescu}) is given in
  the next theorem.
  \begin{thm}[Popescu]
  \label{invariant subspaces}
  If $\s\subset \Fockd$ is invariant for each $S_1,\dots ,S_d$
  then there exists a sequence $\{\phi_j\}_{j\in J}$ of orthogonal
  inner functions such that
  \[
   \s= \oplus_{j\in J}\Fockd \otimes \phi_j.
  \]
  Moreover, this representation is essentially unique.
  \end{thm}

  The model for pure non-commuting
  $d$-tuple is the compression of creation operators
  to a co-invariant subspace as we state next (see \cite{Pop1989}).
  \begin{thm}[Popescu]
  \label{model}
  Let $T$ be a pure non-commuting $d$-tuple.
  Then $T\cong \big(P_{\q}(S_1\ot I_{\ds{1}})|_{\q},\dots,
  P_{\q}(S_d\ot I_{\ds{1}})|_{\q}\big)$,
  where $\q$ is a co-invariant subspace for the creation
  tuples $(S_1\ot I_{\ds{1}},\dots,S_d\ot I_{\ds{1}})$,
  $\ds{1}$ is the first defect space of $T$ and $P_{\q}$ denotes
  the projection on to $\q$.
  \end{thm}
 The co-invariant subspace appears in the above theorem
 is the image of the Poisson kernel $K(T)$ corresponding
 to the tuple $T$ defined by $ K(T): \Hil\to \Fockd\otimes \ds{1},\ h\mapsto
 (\xi_0,\xi_1,\dots)$ where $\xi_0=\omega\otimes D_Th$ and
 for $k\ge 1$,
 \[
  \xi_k=\sum_{f\in F(k,\Lambda)}e_{f}\otimes D_T(T_f)^* h.
 \]In this case, $K(T)$ is an isometry. Moreover, \[K(T) T_i^* =
 (S_i^* \otimes I_{\mathcal{D}_1}) K(T),\]for all $1 \leq i \leq n$,
 and
 \begin{equation}
 \label{adjoint of poisson}
 K(T)^*: \Fockd\otimes \ds{1}\to \Hil,\quad
 e_{f}\otimes \xi\mapsto T_{f}D_T\xi.
 \end{equation}

 For the class of pure tuples the characterization of maximality
 is given in the next theorem.
 \begin{thm}
 \label{characterisation of pure tuple}
 Let $T$ be a pure contractive $d$-tuple of operators on an infinite dimensional
 Hilbert space $\Hil$ with $\dt=1$. Then the following are equivalent:\\
 \textup{(i)} $T$ is maximal.\\
 \textup{(ii)} There is no polynomial $P$ of $d$ non-commuting variables
 such that
 $P(T)|_{\ds{1}}=0$.\\
 \textup{(iii)} $T\cong (P_{\q}S_1|_{\q},\dots, P_{\q}S_d|_{\q})$,
 where $\q$ is the co-invariant subspace of the creation tuple such that
 $\dim[\ran P_{\q}|_{\Gamma_{n]}}]=1+d+\dots+d^{n}$ for all $n\in\Nat$.\\
 \textup{(iv)} For any $n\in \Nat$, $(\Gamma_{n]} \ot \ds{1}) \cap \ker K(T)^*=\{0\}$
 where $K(T)^*$ is the adjoint of the
 Poisson kernel as in ~\eqref{adjoint of poisson}.
 \end{thm}

 \begin{proof}
  (i) $\Leftrightarrow$ (ii) follows from Theorem~\ref{characterisation of tuple}.\\
  (i) $\Leftrightarrow$ (iii)\\
  It is follows from Theorem~\ref{model} that
  $T\cong (P_{\q}S_1|_{\q},\dots, P_{\q}S_d|_{\q})$ where $\q $ is an co-invariant
  subspace for the creation tuple. Thus it is enough to show that the tuple
  $(P_{\q}S_1|_{\q},\dots, P_{\q}S_d|_{\q})$ is maximal if and only if
  $\dim[\ran(P_{\q}|_{\Gamma_{n]}})]=1+d+\dots+d^{n}$ for all $n$. Now we calculate
  the defect spaces of the tuple as follows:
  \[
  \ds{1}=\ran(P_{\q}-\sum_{i=1}^d P_{\q}S_iS_i^*|_{\q})= P_{\q}P_{\Comp\,\omega}P_{\q}.
  \]
  Since the the first defect dimension is one therefore
  $\ds{1}=\spa\{\xi:=P_{\q}(\omega)\}$.
  Note that
  $P_{\q}S_iP_{\q}(\xi)=P_{\q}S_i(I-P_{\q^{\perp}})(\omega)=P_{\q}(e_i)$ as
  $\q^{\perp}$ is an invariant subspace for each $S_i$, $i=1,\dots,d$.
  By induction argument one can show that for any $k\in\Nat$
  and $f\in F(k,\Lambda)$,
  $P_{q}S_{f(1)}P_{\q}\dots P_{\q}S_{f(k)}\xi=P_{\q}(e_{f}) $.
  Then by Proposition~\ref{representation of defect space},
  \[
  \ds{n}= \spa\{ P_{\q}(e_{f}): f\in F_{n-1]} \}
  \] for  all $n$.
  Therefore $\ds{n}=\ran(P_{\q}|_{\Gamma_{n-1]}})$ and the result.\\
  (ii) $\Leftrightarrow$ (iv)\\
  Since the first defect space is one dimensional then $\ds{1}=\Comp\xi$.
  Now as $\ran \Dt=\ran \Dt^2$ we have $\Dt\xi=\lambda\xi$ for some
  non-zero scalar $\lambda$. By definition of $K(T)^*$ it follows
  that $\sum_{f\in F_{k]}} a_{f}e_f\ot\xi\in \ker K(T)^*$ if and only if
  $P(T)(\xi)=0$ where $P=\sum_{f\in F_{k]}}\lambda a_{f} Z_{f}$,
  $Z_{f}=z_{f(1)}\dots z_{f(k)}$ and $k\in\Nat$. Thus the theorem.

 \end{proof}
 \begin{rems}
 (i) The last equivalent condition in the above theorem is independent
 of the assumption $\dt=1$. More precisely, a pure contractive $d$-tuple
 $T$ with finite $\dt$ is maximal if and only if (iv) holds. To see this
 let $\dt=n$ and $\ds{1}=\spa\{\phi_1,\dots,\phi_n\}$.
 Set $\psi_i:=D_T(\phi_i)$, $i=1,\dots n$. Now as $\ran D_T^2=\ds{1}$ and 
 $\ran D_T=\spa\{\phi_1,\dots,\phi_n\}$ we also have
 $\ds{1}=\spa\{\psi_1,\dots \psi_n\}$. Thus $T$ is maximal 
 if and only if the set of vectors $\{T_{f}\psi_i: f\in F_{k]}, i=1,\dots,n\}$
 are linearly independent for any $k\in\Nat$.
 Then the claim readily follows from the following 
 equivalent conditions: 
 \[
  \sum_{i=1,\dots,n, f\in F_{k]}}a_{f,i}e_f\otimes\phi_i\in \ker K(T)^*
  \Leftrightarrow \sum_{i=1,\dots,n, f\in F_{k]}} a_{f,i}T_{f}\psi_i=0.
 \]
 for all $k\in\Nat$.\\
 (ii) The condition (iii) in the above theorem can be made independent of 
 the assumption $\dt=1$ as follows. If $\dt=k$ then
 $T\cong (P_{\q}(S_1\ot I_{\Comp^k})|_{\q},\dots, P_{\q}(S_d\ot I_{\Comp^k})|_{\q})$
 where $\q$ is a joint co-invariant subspace for the ampliated creation tuple. 
 Then the equivalence condition of maximality in this case is the following:
 \[\dim[\ran P_{\q}|_{\Gamma_{n]}\ot \Comp^k}]=(1+d+\dots+d^{n})k\]
 for all $n\in\Nat$.
 \end{rems}

  By Proposition~\ref{pure}, we know that if $T$ is a single pure
  contraction $T$ on a Hilbert space $\Hil$ then
  $\ds{\infty}=\Hil$. Then for
  every polynomial $p$ such that $p(T)|_{\ds{1}}=0$
  implies $0=T^np(T)\xi=p(T)T^n\xi$ for $\xi\in \ds{1}$ and $n\in\Nat$.
  Now as $\ds{\infty}=\ov{\spa}\{T^n\xi:n\in\Nat, \xi\in\ds{1}\}$
  we have $p(T)=0$. Thus for a single pure contraction $T$,
  \[p(T)|_{\ds{1}}=0\Leftrightarrow p(T)=0\]
  for any polynomial $p$.
  This observation helps us to find connection with minimal function as follows.
  Below we denote by $H^{\infty}(\mathbb{D})$ the
  multiplier algebra of the Hardy space on
  the unit disc $H^2(\mathbb{D})$ .

   \begin{thm}
  Let $T$ be a single pure contraction on
  a Hilbert space $\Hil$ with $\dt=1$.\\
  \textup{(a)} If $\Hil$ is infinite dimensional
  and there is a non-zero function $m\in H^{\infty}(\mathbb{D})$
  such that $m(T)=0$
  then $m$ can not be a polynomial.\\
  \textup{(b)} If $\Hil$ is finite dimensional
  then the degree of the minimal polynomial
  is $\dim\Hil$.
  \end{thm}
  \begin{proof}
   Part (a) follows from the
   above discussion and Theorem~\ref{characterisation of tuple}
   and the fact that $T$ is maximal.
   For part (b) note that the maximality of the operator $T$ implies the
   defect spaces in this case are as follows:
    \[\dt^n=\left\{\begin{array}{rl}
                  n, & n\le \dim\Hil\\
                  \dim\Hil, & n>\dim\Hil
                 \end{array}\right. .
   \]
   Then by the Corollary~\ref{single contraction} we have the
 degree of the minimal polynomial is at least $\dim\Hil$
 and this completes the proof.
  \end{proof}

  \begin{rem}
  It is well known that any single pure contraction
  $T$ on a Hilbert space $\Hil$ with $\dt=1$ is unitarily equivalent to
  $P_{H_{\theta}}M_{z}|_{H_{\theta}}$ where
  $H_{\theta}= H^2(\mathbb{D}) \ominus \theta H^2(\mathbb{D})$
  is a co-invariant subspace for the co-ordinate multiplication operator
  $M_{z}$ and $\theta$ is an inner function.
  In this case the minimal function of $T$ is $\theta$
  (see ~\cite{book}, Chapter 3, Proposition 4.3). Then the
  above theorem tells us that if $\Hil$ is infinite dimensional then
  $\theta$ can not be polynomial and if $\theta$ is
  a polynomial then the dimension of $\Hil$ is indeed same as the
  degree of $\theta$.
  \end{rem}

 \section{Maximal submodules of $H^2_d$}

 This section concerns the maximality of submodules
 of the Drury-Arveson module (\cite{D}, \cite{arveson}). We denote by $H^2_d$
 the Drury-Arveson module on the
 unit ball $\mathbb{B}_d$ and defined by the reproducing kernel
 $K_{\lambda} (z) = \frac{1}{(1-<\lambda,z>)},$ where
 $<\lambda,z> = \sum _{j=1} ^d z_j \overline{\lambda _j}$ and
 $\lambda, z\in\mathbb{B}_d$. For $d=1$, $H^2_1=H^2(\mathbb{D})$ the Hardy space
 on the unit disc.
 The multiplication operators $M_{z_i}$ by the co-ordinate functions $z_i$,
 $i=1,\dots,d$,
 turns $H^2_d$ to a Hilbert module over $\mathbb{C}[\bm{z}]:=
 \mathbb{C}[z_1,\dots,z_d]$ as follows:
 \[
  \mathbb{C}[\bm{z}]\times H^2_d \to H^2_d, \
  (p,h)\mapsto p(M_{z_1},\dots,M_{z_d})h.
 \]
  A closed subspace $\s$ of $H^2_d$ is said to be \textit{submodule}
  of $H^2_d$ if $M_{z_i} \s \subseteq \s$ for all $i = 1,
  \ldots, d$. Let $\s$ be a submodule of $H^2_d$ and
  let $R_{\s}:= (M_{z_1} |_{\s},\cdots,M_{z_d} |_{\s})$
  be the restriction of the $d$-shift to $\s.$ It is readily follows that the
  $d$-tuple $R_{\s}$ is contractive. A submodule $\s$
  of $H^2_d$ is \emph{maximal}
  if the contractive tuple $R_{\s}$ is maximal.
 For $d=1$, let $\s\subset H^2_d$ be a submodule
 of the Hardy space on the unit disc. Then $R_{\s}:=M_z|_{\s}$
 is a pure isometry ( $R_{\s}^{*n}\to 0$ in S.O.T. as $n\to \infty$) with multiplicity
 one. In other words, that $R_{\s}\cong M_{z}$. Consequently
 we have the following result:
 \begin{thm}
  Any submodule $\s$ of $H^2(\mathbb{D})$ is maximal.
 \end{thm}

 But for $d\ge 2$ the above theorem does not hold
 in general as we show next.
 For the \emph{rest of the section} we assume $d\ge 2$.

Before proceeding, we shall recall a result concerning the defect
space and the multipliers of submodules of the Drury-Arveson
module(for
   details see
   ~\cite{arveson}, ~\cite{greene}, ~\cite{trent}).
First note that the defect operator $D_{R_{\s}}$ and the defect
dimension of the tuple $R_{\s}$ are given by \[D_{R_{\s}} = (P_{\s}
- \sum _{i=1} ^d M_{z_i} P_{\s}
 M_{z_i}^*)^{1/2},\]
and \[\Delta_{R_{\s}}=\dim [\overline{\mbox{ran}} D_{R_{\s}}],\] where
$P_{\s}$ is the orthogonal projection in $B(H^2_d)$ with range $\s$.

  \begin{thm}
  \label{generator}
  Let $\s$ be a submodule of $H^2_d$ with $\Delta_{T_{\s}}=n$.
  Then there exists
   $\phi_i \in \ran D _{R_{\s}} , i=1,\dots, n$ such that
   each $\phi_i$ is a multiplier and
   $$P_{\s} =\sum_{i=1}^n M_{\phi_i} M_{\phi_i}^*$$
   and the submodule $\s$ is generated by $\{\phi_i\}_{i=1}^n.$
  \end{thm}


    By the above theorem one can describe all the defect spaces
    of $R_{\s}$ in terms of the generators of $\s$
    as follows. Let $\Delta_{R_{\s}}=n$ and $\phi_i, i=1,\dots,n$,
    are as above such that $P_{\s}=\sum_{i}M_{\phi_i}M_{\phi_i}^*$.
    Then

\[\begin{split}
       D_{R_{\s}}^2 & =  P_{\s} -\sum_{i=1}^d M_{z_i} P_{\s} M_{z_i}^*
       \\ & = \sum _{k=1}^n M_{\phi_k}
        (I_{H^2_d} - \sum _{i=1}^d M_{z_i}M_{z_i}^*) M_{\phi_k}^* \\
        &= \sum _{k=1}^n M_{\phi_k} |1\rangle\langle1|M_{\phi_k}^*
        \\& = \sum _{k=1}^n |\phi_k\rangle\langle \phi_k|,
\end{split}\]

   where $|f \rangle\langle g|$ denote the rank one operator
   that takes $h$ to $\langle g, h \rangle f$ for all $f, g, h \in H^2_d$.
   Thus \[\ds{1}=\spa\{\phi_i:i=1,\dots,n\},\] and
   by Proposition~\ref{representation of defect space}
   \[\ds{m}=\spa\{ z_1^{j_1}\cdots z_d^{j_d}\phi_i:
   i=1,\dots n\ \text{ and } \sum_{t=1}^d j_t = m-1\},\] for all $m\in\Nat$.

   Now we investigate the question of maximality of a homogeneous
   submodule $\s$ when $ \Delta _{T_{\s}}$ is finite.

\begin{thm}
Suppose $\s$ is a homogeneous submodule of $H^2_d$ with $ \Delta
_{R_{\s}}<\infty$. Then $\s$ is not maximal.
\end{thm}

    \begin{proof}
    Let  $ \Delta _{R_{\s}} = n$.
    Note that a submodule is homogeneous if and only if it is generated by
    homogeneous polynomials. Consequently, there exists
    an orthonormal basis of $\s$ consisting of homogeneous
    polynomials, and hence there exists polynomials $p_{i},$ $i=1,\dots,n$
    such that $\ds{1}=\spa\{p_{i}:i=1,\dots,n\}$. For the contradiction
    suppose $\s$ is maximal then by maximality of the tuple
    $R_{\s}$, the set of vectors $\{ z_1^{j_1}\cdots z_d^{j_d}p_i: i=1,\dots,n
    \ \text{ and } j_1,\dots, j_d\in \mathbb{N}\}$
    are linearly independent. However since $p_k$'s are polynomials
    those vectors can not be linearly
   independent. This concludes the proof.
    \end{proof}

  \begin{cor}
  Let $\s$ be a submodule of $H^2_d$ with $\Delta_{R_{\s}}<\infty$
  and $P_{\s}=\sum_{i=1}^n M_{p_i}M_{p_i}^*$ for non-constant polynomials
  $p_i$'s and $n\in\Nat$. Then $\s$ is not maximal.
  \end{cor}
  \begin{proof}
  Since $P_{\s}=\sum_{i=1}^n M_{p_i}M_{p_i}^*$,
  the first defect space $\ds{1}=\spa\{p_i: 1 \leq i \leq n\}$.
  Now the argument used to prove the previous theorem can be adapted to
  show that $\s$ is not maximal.
  \end{proof}

  Since $R_{\s}$ is a pure contractive $d$-tuple for a submodule $\s$,
  the adjoint of the Poisson kernel $K(R_{\s})$ 
  in this case is a unique bounded linear operator
   $ K(R_{\s})^*: H^2_d \ot \ds{1}\to H^2_d$ defined
   by taking linear and continuous extension
   of the following prescription:
   \begin{equation*}
    p\ot\xi\mapsto p D_{R_{\s}}\xi,\quad (p\in \mathbb{C}[\bm{z}], \xi\in \ds{1}).
   \end{equation*}
   The range of this map is precisely $\s$.
  A characterization for maximal submodules in terms of
  this operator is given next.

  \begin{thm}
  Let $\s $ be a submodule of $H^2_d$ and $\Delta_{R_{\s}}<\infty$.
    Then the following are equivalent:\\
  \textup{(i)} $\s$ is maximal.\\
  \textup{(ii)} $(\mathbb{C}[\bm{z}]\otimes \ds{1}) \cap \ker K(R_{\s})^* = \{0\}$,
  where the operator $K(R_{\s})^*$ is as above.

  \end{thm}

  \begin{proof}
 The proof follows from a slight modification of the argument given in the 
 first remark after Theorem~\ref{characterisation of pure tuple}
 as the tuple $R_{\s}$ in this case is a commuting tuple.

  \end{proof}
  We conclude the paper with the comment that
  examples of proper maximal submodules for $d\ge 2$
  are not known. We feel that any proper submodule of
  Drury-Arveson module is not maximal but we do not
  have any proof of it
  yet.
  
  \vspace{.2in}
 
\noindent {\bf Acknowledgement:} The hospitality of Indian Statistical Institute,
 Bangalore centre is warmly and gratefully acknowledged by the first author.
 The third author was supported by UGC Centre for Advanced Study.

\end{document}